\documentclass[leqno,12pt,a4paper]{amsart}
\usepackage{amscd,setspace,amssymb,amsopn,amsmath,amsthm,mathrsfs,lmodern,graphics,amsfonts,enumerate,verbatim,calc,times} 
\usepackage[all]{xy}
\usepackage[para]{threeparttable}
\usepackage[full]{textcomp}
\usepackage{newpxtext}
\usepackage{cabin}
\usepackage[varqu,varl]{inconsolata}
\usepackage[bigdelims,vvarbb]{newpxmath}
\usepackage[backend=biber, style=alphabetic, sorting=anyt]{biblatex}
\renewbibmacro{in:}{}
\usepackage[colorlinks=true, citecolor=RoyalBlue, filecolor=RoyalBlue, linkcolor=RoyalBlue,urlcolor = RoyalBlue, hyperindex]{hyperref}
\usepackage{xurl}
\hypersetup{breaklinks=true}
\usepackage[dvipsnames]{xcolor}
\usepackage[referable]{threeparttablex}
\usepackage{todonotes}
\usepackage{footnote}
\usepackage{tikz-cd}
\usepackage{pgfplots}
\pgfplotsset{compat=1.18} 
\usetikzlibrary{spy}
\usepackage[OT2,OT1]{fontenc}

\newcommand\cyr{%
\renewcommand\rmdefault{wncyr}%
\renewcommand\sfdefault{wncyss}
\renewcommandmnj      mm   
    \encodingdefault{OT2}%
\normalfont
\selectfont}
\DeclareTextFontCommand{\textcyr}{\cyr}
\usepackage{amssymb,amsmath}
\DeclareFontFamily{OT1}{rsfs}{}	
\DeclareFontShape{OT1}{rsfs}{n}{it}{<-> rsfs10}{}
\DeclareMathAlphabet{\fmathscr}{OT1}{rsfs}{n}{it}
\numberwithin{equation}{section}
\usepackage[capitalise]{cleveref}
\newtheorem{Theoremx}{Theorem}
 
\newtheorem{theorem}{Theorem}[section]
\theoremstyle{definition}

\theoremstyle{definition}

\newtheorem{lemma}[theorem]{Lemma}

\newtheorem{proposition}[theorem]{Proposition}
\newtheorem{corollary}[theorem]{Corollary}

\newtheorem{conjecture}[theorem]{Conjecture}
\theoremstyle{definition}
\newtheorem{definition}[theorem]{Definition}
\newtheorem{remark}[theorem]{Remark}
\theoremstyle{remark}

\usepackage{stmaryrd}

\newcommand{\ehk}{\operatorname{e}_{\operatorname{HK}}}

\newcommand{\lcm}{\operatorname{lcm}}

\DeclareMathOperator{\hks}{\mathbf{HKS}}

\newcommand{\Hom}{\operatorname{Hom}}

\newcommand{\Char}{\operatorname{char}}

\newcommand{\N}{\mathbb{N}}
\newcommand{\Z}{\mathbb{Z}}
\newcommand{\F}{\mathbb{F}}

\newcommand{\R}{\mathbb{R}}

\newcommand{\fm}{\mathfrak{m}}

\newcommand{\fa}{\mathfrak{a}}

\newcommand{\FF}{\mathbb{F}}

\makeatletter
\renewcommand{\subsection}{%
  \@startsection{subsection}
    {2}
    {\z@}
    {-21dd plus-8pt minus-4pt}
    {10.5dd}
    {\normalsize\bfseries\boldmath}%
}
\makeatother

\begin{document}

\title{ACC for $F$-signature: A likely counterexample}
\thanks{The SMALL REU was funded by NSF Grants DMS \#2241623 and DMS \#1947438}

\author[Adams]{Clay Adams}
\address{Department of Mathematics, Harvey Mudd College, Claremont, CA 91711}
\email{ccadams@hmc.edu}

\author[Sandstrom]{Theodore J. Sandstrom}
\address{Department of Mathematics, Statistics, and Computer Science, University of Illinois at Chicago, Chicago, IL 60607}
\email{tsands3@uic.edu}

\author[Simpson]{Austyn Simpson}
\thanks{Simpson was supported by NSF postdoctoral fellowship DMS \#2202890.}
\address{Department of Mathematics, University of Michigan, Ann Arbor, MI 48109 USA}
\email{austyn@umich.edu}
\urladdr{\url{https://austynsimpson.github.io}}

\begin{abstract}
Let $\mathscr{k}=\overline{\F_2}$ and let $0\neq\alpha\in \mathscr{k}$. We present a conjecture supported by computer experimentation involving the Brenner--Monsky quartic $g_\alpha=\alpha x^2y^2+z^4+xyz^2+(x^3+y^3)z\in \mathscr{k}\llbracket x,y,z\rrbracket$. If true, this conjecture provides a formula for the Hilbert--Kunz multiplicity and $F$-signature of the family of four-dimensional hypersurfaces defined by $uv+g_\alpha\in \mathscr{k}\llbracket x,y,z,u,v\rrbracket$ which depends on $[\F_2(\alpha):\F_2]$, giving an infinite increasing chain of strict inequalities of $F$-signatures. Additionally, we obtain for any $t\in\N$ a formula for the Hilbert--Kunz multiplicity and $F$-signature of the $t$--parameter family of $3t+1$--dimensional hypersurfaces defined by $uv+\sum\limits_{i=1}^t g_{\alpha_i}(x_i,y_i,z_i)$.

\end{abstract}
\maketitle

\section{Introduction}

Let $(R,\fm,\mathscr{k})$ be a complete local $F$-finite domain of prime characteristic $p>0$ and Krull dimension $d$. The \emph{$F$-signature} of $R$ is a numerical invariant which measures the free rank of $R^{1/p^e}$ as $e$ grows by comparing it to that of $S^{1/p^e}$ where $S$ is a regular local ring of the same dimension. One precise formulation of this invariant comes from considering a certain sequence of $\fm$-primary ideals, the so-called \emph{Frobenius degeneracy ideals}
 \begin{align*}
     I_e:=&\langle r\in R\mid R\stackrel{1\mapsto r^{1/p^e}}{\longrightarrow} R^{1/p^e}\text{ is not a split }R\text{-module inclusion}\rangle\\
     =&\langle r\in R\mid \phi(F^e_* r)\in\fm \text{ for all }\phi\in\Hom_R(R^{1/p^e},R)\rangle.
 \end{align*}
The $F$-signature is then defined as the limit $s(R):=\lim\limits_{e\rightarrow\infty}\frac{\ell_R(R/I_e)}{p^{ed}}$ which was shown to exist in \cite{Tuc12}. This value captures many delicate properties of the ring; namely there are inequalities $0\leq s(R)\leq 1$ with $s(R)=1$ if and only if $R$ is regular \cite[Corollary 16]{HL02}. Moreover, $s(R)>0$ if and only if $R$ is \emph{strongly $F$-regular} \cite[Theorem 0.2]{AL03} --- a prominent class of mild singularities in prime characteristic which is thought to be the counterpart of the klt singularities of the complex minimal model program.

The present article is concerned with the existence of a countably infinite sequence of elements $\{h_n\}$ in a power series ring $S=\mathscr{k}\llbracket x_1,\dots, x_d\rrbracket$ such that $$s(S/(h_1))<s(S/(h_{2}))<\cdots$$ where all inequalities are strict. It is conjectured after \cite[Example 4.9]{PS20} that such a family exists in light of a family demonstrating similar behavior for the \emph{Hilbert--Kunz multiplicity,} a related invariant given by the limit $\ehk(R):=\lim\limits_{e\rightarrow\infty}\frac{\ell_R(R/\fm^{[p^e]})}{p^{ed}}$. This limit exists by \cite{Mon83}, and Monsky showed in \cite{Mon98} that if $\alpha_n\in \overline{\F_2}$ is a collection of elements such that $[\F_2(\alpha_n):\F_2]=n$ and $$g_{\alpha_n}=(\alpha_n^2+\alpha_n)x^2y^2+z^4+xyz^2+(x^3+y^3)z\in S:=\overline{\F_2}\llbracket x,y,z\rrbracket$$ then $\ehk(S/(g_{\alpha_2}))>\ehk(S/(g_{\alpha_3}))>\cdots$ where all inequalities are strict. In fact, it was shown that $\ehk(R/(g_{\alpha_n}))=3+4^{-n}$, and this family was subsequently used in \cite{BM10} to show that tight closure need not localize. 

Since the hypersurfaces defined by the $g_{\alpha_n}$ are not strongly $F$-regular, more work is required in finding a similar example for the $F$-signature. We repurpose a smoothening idea which Monsky used in \cite{Mon08} in his search for \emph{irrational} Hilbert--Kunz multiplicities, which turns out to produce $F$-regular examples. Roughly, the strategy is to relate the Hilbert--Kunz function of the five variable hypersurface $uv+g_\alpha$ to that of $g_\alpha^j$ for a range of values for $j>0$. $uv+g_\alpha$ then defines a 
four dimensional $F$-regular hypersurface of multiplicity two, a scenario in which the values of $\ehk(-)$ and $s(-)$ are known to differ by a constant (see \cite{HL02}).

This method involving adjoining variables further exacerbates the difficulty of computing the invariant, and much like \cite{Mon08} our formulae are only conjectural with some supporting evidence from computer experiments. Our conjecture (see \cref{conj:bracket}) suggests a relationship between the values $e_n(g_\alpha^j)$ (i.e. the quantity $\dim_{\overline{\F_2}}\frac{\overline{\F_2}[x,y,z]}{(x^{2^n},y^{2^n},z^{2^n},g_\alpha^j)}$) and the output of a certain family of dynamical systems introduced in \cref{def: bracket}. Such functions are inspired by Monsky and Teixeira's work on $p$-fractals \cite{MT04,MT06}. We obtain the following conclusion:

\begin{Theoremx}\label{thm:theorem-A}(= \Cref{thm:F-sig})
Let $\mathscr{k}=\overline{\F_2}$, and for each $0\neq\alpha\in \mathscr{k}$ define $m_\alpha:=[\F_2(\alpha):\F_2]$. Let $R_\alpha=\mathscr{k}\llbracket x,y,z,u,v\rrbracket/(G_\alpha)$ where $G_\alpha=uv+\alpha x^2 y^2+z^4+xyz^2+(x^3+y^3)z$. If \Cref{conj:bracket} is true, then for $m_\alpha\geq 2$ we have
$$\ehk(R_\alpha)=\frac{45\cdot 2^{3m_\alpha}-38}{7\cdot 2^{3m_\alpha+2}-28}$$ and 
$$s(R_\alpha)=\frac{11\cdot 2^{3m_\alpha}-18}{7\cdot 2^{3m_\alpha+2}-28}.$$
\end{Theoremx}

The condition that $m_\alpha\geq 2$ is only present here for simplicity -- we propose values for the invariants at $m_\alpha=1$, but this case is handled with a separate argument in \cref{subsection:m equals 1}. Assuming the same conjecture, we are also able to extract a formula for the Hilbert--Kunz multiplicity and $F$-signature which varies in a multi-parameter family. For simplicity we present the following special case for two parameters:

\begin{Theoremx}(= \Cref{thm: multi-parameter})
Let $\mathscr{k}=\overline{\F_2}$ and let $\alpha,\beta\in \mathscr{k}\setminus\{0\}$ be two elements such that $m_\alpha=[\FF_2(\alpha):\FF_2]>1$ and $m_\beta=[\FF_2(\beta):\FF_2]>1$. Let $g_\alpha$ and $g_\beta$ be the Brenner--Monsky quartics in the disjoint sets of variables $x_1,y_1,z_1$ and $x_2,y_2,z_2$ respectively. If \Cref{conj:bracket} is true, then the seven dimensional hypersurface $R_{\alpha,\beta}:=\frac{\mathscr{k}\llbracket x_1,y_1,z_1,x_2,y_2,z_2,u,v\rrbracket }{(uv+g_\alpha+g_\beta)}$ has Hilbert--Kunz multiplicity
\begin{align*}
    \ehk(R_{\alpha,\beta})=\frac{381}{248}+\frac{3}{2^{5m_\alpha+3}-8}+\frac{3}{2^{5m_\beta+3}-8}+\frac{1}{2^{5\lcm(m_\alpha,m_\beta)+3}-8}
\end{align*}
and $F$-signature
\begin{align*}
    s(R_{\alpha,\beta})=\frac{115}{248}-\frac{3}{2^{5m_\alpha+3}-8}-\frac{3}{2^{5m_\beta+3}-8}-\frac{1}{2^{5\lcm(m_\alpha,m_\beta)+3}-8}.
\end{align*}
\end{Theoremx}

\subsection*{Acknowledgments}
We are grateful to Kevin Tucker for helpful conversations. We thank Ilya Smirnov for helpful comments on an earlier draft, particularly for suggestions which sped up our calculations in Macaulay2. Part of this project was completed at the 2023 SMALL REU at Williams College, which we thank Steven J. Miller for organizing.

\section{The conjecture}
We begin by fixing notation that will be used throughout the article. Fix an algebraic closure $k:=\overline{\F_2}$. For a scalar $\beta\in \mathscr{k}$ we define $m_\beta:=[\F_2(\beta):\F_2]$. For a polynomial $f\in \mathscr{k}[x_1,\dots, x_d]$ and $n\geq 0$ we let $e_n(f)$ be the $n$-th Hilbert--Kunz number of $f$, that is $$e_n(f)=\dim_{\mathscr{k}}\frac{\mathscr{k}[x_1,\dots, x_d]}{(x_1^{2^n},\dots, x_d^{2^n},f)}.$$ In the sequel, $f$ will typically be either a power of $g_\alpha:=\alpha x^2y^2+z^4+xyz^2+(x^3+y^3)z\in \mathscr{k}[x,y,z]$ or $G_\alpha:=uv+g_\alpha\in \mathscr{k}[x,y,z,u,v]$, where $0\neq\alpha\in \mathscr{k}$ varies.

We now define a family of dynamical systems (depending on $\alpha$) which predicts the values of $e_n(g_\alpha^j)$.

\begin{definition}\label{def: bracket}
    Let $0\neq \alpha\in \mathscr{k}$ and write $\alpha=\lambda^2+\lambda$ for some $\lambda\in \mathscr{k}$. Let $\Gamma$ be the free abelian group on symbols $A_n, B_n, C, D$ (with identity element denoted by $0$) where $n$ takes values in $\N$. Define the morphisms $\sigma_0, \sigma_1:\Gamma\to\Gamma$ by the following procedure:
    \begin{align*}
    \underline{\text{If }m_\lambda\neq m_\alpha:} &\\
    \sigma_0(A_n) &= \begin{cases}
        B_{n+1} & m_\alpha\mid n+1,\ \frac{n+1}{m_\alpha}\textrm{ is even,}\\
        A_{n+1} & \textrm{otherwise}
    \end{cases} & 
    \sigma_1(A_n) &= \begin{cases}
        B_{n+1} & m_\alpha\mid n+1,\ \frac{n+1}{m_\alpha}\textrm{ is odd,}\\
        A_{n+1} & \textrm{otherwise}
    \end{cases}\\
    \sigma_0(B_n) &= \begin{cases}
        0 & m_\alpha\mid n,\ \frac{n}{m_\alpha}\textrm{ is even,}\\
        A_{m+1}+C & \textrm{otherwise}
    \end{cases} & 
    \sigma_1(B_n) &= \begin{cases}
        0 & m_\alpha\mid n,\ \frac{n}{m_\alpha}\textrm{ is odd,}\\
        A_{n+1}+C & \textrm{otherwise}
    \end{cases} \\
    \sigma_0(C) &= C &
    \sigma_1(C) &= C \\
    \sigma_0(D) &= A_1 &
    \sigma_1(D) &= 0.
\end{align*}
\begin{align*}
    \underline{\text{If }m_\lambda= m_\alpha:} &\\
    \sigma_0(A_n) &= \begin{cases}
        B_{n+1} & m_\alpha\mid n+1,\ \frac{n+1}{m_\alpha}\textrm{ is odd,}\\
        A_{n+1} & \textrm{otherwise}
    \end{cases} & 
    \sigma_1(A_n) &= \begin{cases}
        B_{n+1} & m_\alpha\mid n+1,\ \frac{n+1}{m_\alpha}\textrm{ is even,}\\
        A_{n+1} & \textrm{otherwise}
    \end{cases}\\
    \sigma_0(B_n) &= \begin{cases}
        0 & m_\alpha\mid n,\ \frac{n}{m_\alpha}\textrm{ is odd,}\\
        A_{m+1}+C & \textrm{otherwise}
    \end{cases} & 
    \sigma_1(B_n) &= \begin{cases}
        0 & m_\alpha\mid n,\ \frac{n}{m_\alpha}\textrm{ is even,}\\
        A_{n+1}+C & \textrm{otherwise}
    \end{cases} \\
    \sigma_0(C) &= C &
    \sigma_1(C) &= C \\
    \sigma_0(D) &= A_1 &
    \sigma_1(D) &= 0.\\
    \end{align*}
Note that we are suppressing the dependence on $\alpha$ in the notation $\sigma_i$. Then for any $n \geq 0$ and any $0 \leq j < 2^n$ we define an element $f(n, j)$ recursively by
    \begin{align*}
        f(0,0) &= 2C + D, & f(n+1, 2j) &= \sigma_0(f(n,j)), & f(n+1, 2j+1) &= \sigma_1(f(n,j)).
    \end{align*}
    Notice that each application of $\sigma_0$ or $\sigma_1$ increments the subscript of the $A_n$'s and $B_n$'s. This, with the action of $\sigma_0$ on the initial symbol $D$, shows that $f(n, j)$ will always be a linear combination of $A_n$, $B_n$, $C$ and $D$ -- say $f(n, j)=aA_n + bB_n + cC + dD$. We define $\langle n, j,\alpha\rangle:=(3a + 5b + d)2^c$.
\end{definition}

\begin{conjecture}\label{conj:bracket}
 For $n\geq 1$, $j\geq 1$, and $\alpha\in \mathscr{k}\setminus \{0\}$ we have $$\langle n,j,\alpha\rangle=e_{n+1}\left(g_\alpha^{2j+1}\right)-\frac{1}{2}\left(e_{n+1}\left(g_\alpha^{2j}\right)+e_{n+1}\left(g_\alpha^{2j+2}\right)\right).$$
 \end{conjecture}

 \begin{remark}
\begin{enumerate}
    \item \cref{conj:bracket} should be compared with \cite[Conjecture 1.5]{Mon08}, which gives a similar flavor of prediction for the values of $e_n(H^j)$ where $H$ is the nodal cubic $x^3+y^3+xyz\in\F_2[x,y,z]$. Using this, Monsky conjectures that $uv+H$ defines a hypersurface with Hilbert--Kunz multiplicity $\frac{4}{3}+\frac{5}{14\sqrt{7}}$ \cite[Corollary 2.7]{Mon08}.
    \item We have verified via Macaulay2 \cite{M2} that \Cref{conj:bracket} is true for all $1\leq n\leq 6$ and for all elements $\alpha$ living in extensions $L\supseteq \F_2$ with $[L:\F_2]\leq 9$. We have also verified the conjecture for $n=7,8$ but only for a small sample of representative elements $\alpha$ for each given $1\leq m_\alpha\leq 8$ due to how computationally expensive this task is.
\end{enumerate}     
 \end{remark}

\begin{lemma}\label{lemma-eq} Let $0\neq \alpha\in \mathscr{k}$ and $n\geq 1$. Then:
 \begin{enumerate}
     \item $e_n(uv+g_\alpha)=2\sum\limits_{j=1}^{2^n -1} e_n\left(g_\alpha^j\right)+e_n\left(g_\alpha^{2^n}\right)$;\label{lemma-1}
     \item For all $j\geq 2^{n-1}$, we have 
         $e_{n+1}\left(g_\alpha^{2j+1}\right)=\frac{1}{2}\left(e_{n+1}\left(g_\alpha^{2j}\right)+e_{n+1}\left(g_\alpha^{2j+2}\right)\right)$;\label{lemma-2}
     \item For all $j\geq 0$ we have $e_{n+1}\left(g_\alpha^{2j}\right)=8e_n\left(g_\alpha^j\right)$.\label{lemma-3}
     \item $\langle n,j,\alpha\rangle=0$ for all $2^{n-1}\leq j\leq 2^n -1$.\label{lemma-4}
 \end{enumerate}
 \end{lemma}
 \begin{proof}
 (\ref{lemma-1}) follows from the same proof as \cite[Theorem 2.2]{Mon08}, so we omit it. To see (\ref{lemma-2}), let $j<2^n$. Observe that since $g_\alpha$ is homogeneous of degree four,
 \begin{align*}
e_{n+1}(g_\alpha^{2^n+j})=\dim_{\mathscr{k}}\frac{\mathscr{k}[x,y,z]}{(x^{2^{n+1}},y^{2^{n+1}},z^{2^{n+1}},g_\alpha^{2^n+j})}=\dim_{\mathscr{k}}\frac{\mathscr{k}[x,y,z]}{(x^{2^{n+1}},y^{2^{n+1}},z^{2^{n+1}})}=2^{3n+3}
\end{align*}
which is independent of $j$. (\ref{lemma-3}) follows from the fact that $\mathscr{k}[x,y,z]$ is free of rank $8$ over $\mathscr{k}[x^2,y^2,z^2]$. We leave (\ref{lemma-4}) as an exercise for the reader.
 \end{proof}
We remark that \cref{lemma-eq}(\ref{lemma-2}) and \cref{lemma-eq}(\ref{lemma-4}) confirm \cref{conj:bracket} in the case of $2^{n-1}\leq j\leq 2^n -1$.

 \begin{lemma}\label{lemma:recurrence}
If \cref{conj:bracket} is true, then for all $n\geq 0$, 
 \begin{align}
 e_{n+1}(uv+g_\alpha)-16e_n(uv+g_\alpha)=2\sum\limits_{j=0}^{2^n-1}\langle n,j,\alpha\rangle.\label{eqn:recurrence}
 \end{align}
 If in addition $n\geq 1$, then
 \begin{align}
 e_{n+1}(uv+g_\alpha)-16e_n(uv+g_\alpha)=2\sum\limits_{j=0}^{2^{n-1}-1}\langle n,j,\alpha\rangle.
 \end{align}
 \end{lemma}
 \begin{proof}
 Observe that
 {\scriptsize
 \begin{align}
     e_{n+1}(uv+g_\alpha)-16e_n(uv+g_\alpha)&=2\sum\limits_{j=1}^{2^{n+1} -1} e_{n+1}\left(g_\alpha^j\right)+e_{n+1}\left(g_\alpha^{2^{n+1}}\right)-32\sum\limits_{j=1}^{2^n -1} e_n\left(g_\alpha^j\right)-16e_n\left(g_\alpha^{2^n}\right)\label{eqn:recurrence-1}\\
     &=2\sum\limits_{j=0}^{2^n -1}e_{n+1}\left(g_\alpha^{2j+1}\right)+2\sum\limits_{j=1}^{2^n -1} e_{n+1}\left(g_\alpha^{2j}\right)+e_{n+1}\left(g_\alpha^{2^{n+1}}\right)-32\sum\limits_{j=1}^{2^n-1}e_n\left(g_\alpha^j\right)-16e_n\left(g_\alpha^{2^n}\right)\nonumber\\
     &=2\sum\limits_{j=0}^{2^n-1}\langle n,j,\alpha\rangle+8\sum\limits_{j=0}^{2^n-1}\left(e_n\left(g_\alpha^j\right)+e_n\left(g_\alpha^{j+1}\right)\right)-16\sum\limits_{j=1}^{2^n-1} e_n\left(g_\alpha^j\right)-8e_n\left(g_\alpha^{2^n}\right)\label{eqn:recurrence-2}\\
     &=2\sum\limits_{j=0}^{2^n-1}\langle n,j,\alpha\rangle\label{eqn:recurrence-3}
 \end{align}
 }
 where \cref{eqn:recurrence-1} follows from \cref{lemma-eq}(\ref{lemma-1}) and \cref{eqn:recurrence-2} from \cref{lemma-eq}(\ref{lemma-3}). If $n\geq 1$, the quantity in \cref{eqn:recurrence-3} is equal to $2\sum\limits_{j=0}^{2^{n-1}-1}\langle n,j,\alpha\rangle$ by \cref{lemma-eq}(\ref{lemma-2}).
 \end{proof}

As we are interested in computing (after completion) the limit $\lim\limits_{n\rightarrow\infty}\frac{e_n(uv+g_\alpha)}{16^n}$, we next aim to calculate the sum $\sum\limits_{j=0}^{2^n-1}\langle n,j,\alpha\rangle$. \cref{subsection:m greater than 1,subsection:m equals 1} handle this separately for $m_\alpha>1$ and $m_\alpha=1$ respectively, and \cref{subsection:ehk and s} extracts the relevant limits for $\ehk(R_\alpha)$ and $s(R_\alpha)$ where $R_\alpha=\mathscr{k}\llbracket x,y,z,u,v\rrbracket/(G_\alpha)$.

\subsection{\texorpdfstring{$m_\alpha>1$}{m\_α>1}}\label{subsection:m greater than 1}
In this subsection, fix an element $\alpha\in \mathscr{k}$ with $m:=m_\alpha>1$. We will analyze the sequence $$d_{n,\alpha}:=\begin{cases}4:& n=0\\2^{n+3}:& n\geq 1\text{ and }m_\alpha\mid n\\ 3\cdot 2^{n+1}:& n\geq 1\text{ and } m_\alpha\nmid n.\end{cases}$$ We will show that $\sum\limits_{j=0}^{2^n-1}\langle n,j,\alpha\rangle =d_{n,\alpha}$ for all $n\geq 0$, and the next several definitions and lemmas are aimed at this.
\begin{definition}\label{def:ant-bnt}
    For nonnegative integers $n$ and $t$, let $a_{n, t}$ and $b_{n, t}$ denote the following quantities:
\begin{align*}
    a_{n,t}:=\#\{0\leq j<2^n\mid f(n,j)=A_n+tC\}\\
    b_{n,t}:=\#\{0\leq j<2^n\mid f(n,j)=B_n+tC\}.
\end{align*}    
\end{definition}
Notice that $a_{1, 2} = 1$ and that $a_{1,t}=0$ for all $t\neq 2$. \cref{def: bracket,def:ant-bnt} provide to us the identity
\begin{equation}\label{eqn: sum bracket from ant bnt}
    \sum_{j=0}^{2^{n-1}-1}\langle n, j, \alpha\rangle = \sum_t 2^t (3a_{n,t} + 5b_{n,t})
\end{equation}
for all $n > 0$, and we will proceed by studying the double sequences $a_{n,t}$ and $b_{n,t}$.
\begin{remark}
As with $f(n,j)$, note that we are suppressing reference to $\alpha$ in \cref{def:ant-bnt}. This is justified as follows. Suppose that $0\neq \alpha,\alpha'\in \mathscr{k}$ with $m_\alpha=m_{\alpha'}>1$. Further suppose that $\alpha=\lambda^2+\lambda$ and $\alpha'=\lambda'^2+\lambda'$ where $m_\alpha=m_\lambda$ but $m_{\alpha'}\neq m_{\lambda'}$. One benefit of reframing the calculation in terms of \cref{def:ant-bnt} is that while we will have $\langle n,j,\alpha\rangle\neq\langle n,j,\alpha'\rangle$, the values for $a_{n,t}$ and $b_{n,t}$ will not differ for $\alpha$ and $\alpha'$. Hence, we will not treat these cases separately for the remainder of the subsection.
\end{remark}
Additionally, using the morphisms defined in \cref{def: bracket}, we derive the following relations:
\begin{align} 
    a_{n+1,t} &= \begin{cases}
        a_{n,t} & \text{if }m\mid n+1\\
        2a_{n,t} + a_{n, t-1} & \text{if }m\mid n\\
        2a_{n,t} & \text{otherwise,}
    \end{cases} \label{eqn: ant rr}\\
    b_{n,t} &= \begin{cases}
        a_{n, t} & \textrm{if }m \mid n\\
        0 &\textrm{otherwise.}
    \end{cases} \label{eqn: bnt relation}
\end{align}

\begin{lemma}\label{lemma: first nonzero k for ant}
Let $t\geq 2$. Then $\min\{r\mid a_{mr,t}\neq 0\}=t-1$ and $a_{m(t-1),t} = 2^{(m-2)(t-1)}$.
\end{lemma}
\begin{proof}
    We prove both statements simultaneously by induction on $t$. If $t = 2$, then $a_{0,2} = 0$ and $a_{1, 2} = 1$. We then obtain $a_{m, 2}=2^{m-2}$ by applying \cref{eqn: ant rr} repeatedly. Now suppose the lemma holds for some $t \geq 2$. Repeated applications of \cref{eqn: ant rr} give the following relation for all $r\geq 0$:
    \begin{align*}
        a_{mr,t+1}=2^{m-1} a_{m(r-1),t+1}+2^{m-2}a_{m(r-1),t}
    \end{align*}
    If $r \leq t - 1$, then $a_{m(r-1),t}$ is zero by the inductive hypothesis. For these $r$, we find $a_{mr, t+1} = 2^{m-1}a_{m(r-1),t+1}=0$ since $a_{0, t+1} = 0$ by definition. When $r = t$, the first summand, $2^{m-1}a_{m(t-1),t+1}$, is zero by our earlier argument. Thus, we see that $a_{mt, t+1} = 2^{m-2}a_{m(t-1),t} = 2^{(m-2)t}$, as desired.

\end{proof}
\begin{proposition}\label{prop: ant rr for mk t+1}
For all $j\geq 0$, we have the following recurrence relation:
\begin{equation}\label{eqn: ant rr for mk t+1}
    a_{m(t+j), t+1} = \sum_{s=0}^j 2^{(m-1)(j-s+1)-1}a_{m(t+s-1),t}
\end{equation}
\end{proposition}
\begin{proof}
    The proof will proceed by induction on $j$. By \cref{lemma: first nonzero k for ant}, the formula is satisfied when $j = 0$, so suppose that it holds for some $j$. Applying \cref{eqn: ant rr} $m$ times to the right hand side of \cref{eqn: ant rr for mk t+1} shows that:
\begin{align*}
    a_{m(t+j+1), t+1} &= 2^{m-1}\sum_{s=0}^j \left(2^{(m-1)(j-s+1)-1}a_{m(t+s-1),t} \right) + 2^{m-2}a_{m(t+j), t}\\
    &= \sum_{s=0}^{j+1} 2^{(m-1)(j-s+2)-1}a_{m(t+s-1),t}
\end{align*}
as desired.
\end{proof}

\begin{lemma}\label{lemma: formula for t = 2}
    For all $r > 0$, we have $a_{mr, 2} = 2^{(m-1)r-1}$
\end{lemma}
\begin{proof}
    The proof is by induction on $r$. When $r = 1$, the result follows by \cref{lemma: first nonzero k for ant}. If the result holds for some $r$, applying \cref{eqn: ant rr} $m$ times amounts to doubling $a_{mr, 2}$ exactly $m-1$ times, so $a_{m(r+1),2} = 2^{(m-1)(r+1)-1}$, as desired.
\end{proof}

This yields the following description of $a_{mr,t}$:
\begin{proposition}\label{prop: formula for amk t}
    For integers $t\geq 2$ and $r\geq 0$, we have:
    \begin{equation}\label{eqn: formula for amk t}
        a_{mr, t} = 2^{(m-1)r+1-t}\binom{r-1}{t-2}.
    \end{equation}
\end{proposition}
\begin{proof}
    The proof proceeds by induction on $t$. When $t = 2$, the formula reduces to $a_{mr, 2} = 2^{(m-1)r-1}$, which holds by \cref{lemma: formula for t = 2}. Now suppose \cref{eqn: formula for amk t} holds for some $t$. By \cref{lemma: first nonzero k for ant} it suffices to consider terms of the form $a_{m(t+j),t+1}$ for $j\geq 0$.  We apply \cref{prop: ant rr for mk t+1} and simplify:
    \begin{align*}
        a_{m(t+j), t+1} &= \sum_{s=0}^j 2^{(m-1)(j-s+1)-1}a_{m(t+s-1),t}\\
        &= \sum_{s=0}^j 2^{(m-1)(j-s+1)-1}2^{(m-1)(t+s-1)+1-t} \binom{t+s-2}{t-2}\\
        &= 2^{(m-1)(j+t)+1-(t+1)} \sum_{s=0}^j \binom{t+s-2}{t-2}\\
        &=  2^{(m-1)(j+t)+1-(t+1)}\binom{t+j-1}{t-1}.
    \end{align*} This completes the proof.
\end{proof}
Applying \cref{eqn: ant rr} in reverse gives the following corollary:
\begin{corollary}\label{corr: formula for ant bnt}
    For all integers $t\geq 2$, $r\geq 0$ and $0 < i \leq m$, we have:
    \begin{align*}
        a_{m(r-1)+i, t} &= \begin{cases}
            \displaystyle 2^{(m-1)r+1-t}\binom{r-1}{t-2}& \text{if }i = m\\
            \displaystyle 2^{(m-1)r+i+2-t-m}\binom{r-1}{t-2} & \text{if }i < m
        \end{cases}\\
        b_{m(r-1)+i, t} &= \begin{cases}
            \displaystyle 2^{(m-1)r+1-t}\binom{r-1}{t-2} & \text{if }i = m\\
            0 &\text{otherwise.}
        \end{cases}
    \end{align*}
\end{corollary}
With this description, we can compute the sum in \cref{eqn: sum bracket from ant bnt}. 

\begin{theorem}\label{thm:difference-en}
$\sum\limits_{j=0}^{2^n-1}\langle n,j,\alpha\rangle =d_{n,\alpha}$ for all $n\geq 0$.
 \end{theorem}
 \begin{proof}
By \cref{def: bracket}, $\langle 0,0,\alpha\rangle=4$ for all $0\neq \alpha\in \mathscr{k}$. Now suppose that $n\geq 1$ and that $n=mr$ for some $r\geq 1$. Using \cref{eqn: sum bracket from ant bnt}, we can write
    \[\sum_{j=0}^{2^{n-1}-1}\langle n, j, \alpha\rangle = 2^3\sum_{t\geq 2} 2^ta_{n,t}\]
    By \cref{lemma: first nonzero k for ant}, the terms above are zero when $t> r+1$, so we can use
    \cref{corr: formula for ant bnt}:
    \begin{align*}
        \sum_{j=0}^{2^{n-1}-1}\langle n, j, \alpha\rangle = 2^3\sum_{t=2}^{r+1} 2^ta_{n,t}= 2^{(m-1)r+4} \sum_{t=2}^{r+1}\binom{r-1}{t-2}= 2^{mr+3} = 2^{n+3}.
    \end{align*}
    Lastly, we suppose that $n = m(r-1)+i$ for some $r\geq 1$ and $0<i < m$. Then, each $b_{n, t}$ is zero, so we have
    \begin{align*}
    \sum_{j=0}^{2^{n-1}-1}\langle n, j, \alpha\rangle = \sum_{t=2}^{r+1} 2^ta_{n,t}= 2^{(m-1)r+i+2-m} \sum_{t=2}^{r+1} \binom{r-1}{t-2}= 2^{m(r-1)+i+1} = 2^{n+1}.
    \end{align*}
\end{proof}

\subsection{\texorpdfstring{$m_\alpha=1$}{m\_α=1}}\label{subsection:m equals 1}

We now consider the case $m_\alpha = 1$, i.e. when $\alpha=1$. Note first that since $1=\lambda^2+\lambda$ for some $\lambda\in\F_4\setminus\F_2$, we are in the setting of the first pair of morphisms $\sigma_0,\sigma_1$ of \cref{def: bracket} where $m_\alpha=1$ and $m_\lambda=2$. Even though \cref{conj:bracket} holds experimentally in this case, the recurrence relations in \cref{eqn: ant rr} and \cref{eqn: bnt relation} fail. The appropriate replacement is
\begin{align}
    a_{n+1, t} &= a_{n,t} + a_{n-1,t-1}\label{eqn: ant rr for m = 1}\\
    b_{n+1, t} &= a_{n,t}\label{eqn: bnt rr for m = 1}
\end{align}
where $a_{n,t}$ and $b_{n,t}$ are zero if $t<2$, $a_{1, 2} = 1$ and $b_{1,2} = 0$, and $a_{1,t}=b_{1, t}=0$ for all $t\neq 2$. We begin with a lemma which describes when $a_{n,t}\neq 0$.
\begin{lemma}
    For any $t\geq 2$, we have  $\min\{n\geq 1\mid a_{n,t}\neq 0\}= 2t-3$.
\end{lemma}
\begin{proof}
    The result holds when $t = 2$ by definition. For arbitrary $t$, if $n < 2t-3$ then $a_{n,t} = a_{n-1,t}$ which is zero by induction. When $n = 2t-3$, we find that $a_{n,t} = a_{n-2,t-1}\neq 0$.
\end{proof}

\begin{proposition}
    For all $t\geq 2$ and $n \geq 2t-3$, the value of $a_{n,t}$ is given by the binomial coefficient $\binom{n+1-t}{t-2}$.
\end{proposition}
\begin{proof}
    In the case that $t = 2$, the recurrence relation reduces to $a_{n+1,2} = a_{n,2}$, so $a_{n,2}=1$ for all $n\geq 2$; this agrees with the desired formula. If the proposition holds for all $t < t'$ and for all pairs $(n, t')$ where $2t'-3\leq n < n'$, then we have
    \[a_{n',t'} = a_{n'-1,t'} + a_{n'-2,t'-1} = \binom{n'-t'}{t'-2} + \binom{n'-t'}{t'-3} = \binom{n'+1-t'}{t'-2}.\]
    The desired result then follows by induction.
\end{proof}
We will now analyze the sequence $c_n:=\begin{cases}4:& n=0\\\frac{22}32^n + (-1)^n\frac 83:& n>0\end{cases}$. For ease of notation, momentarily let $h_n=\sum\limits_{j=0}^{2^n -1} \langle n,j,\alpha\rangle$ (note that we are suppressing reference to $\alpha$). By \cref{eqn: sum bracket from ant bnt} and \cref{lemma-eq}(\ref{lemma-2}), we get the following equation for all $n\geq 1$:
\begin{equation}\label{eqn: sumbracket m = 1}
 h_n=\sum_{j=0} ^{2^{n-1}-1}\langle n, j, \alpha\rangle = \sum_{t=2}^{\lceil\frac n2\rceil + 1}3\cdot 2^t\binom{n+1-t}{t-2} + \sum_{t=2}^{\lfloor \frac n2 \rfloor +1} 5\cdot 2^t\binom{n-t}{t-2}.
\end{equation}
Using the above formula, we can derive the following recurrence relation.
\begin{proposition}\label{prop: dn rr for m = 1}
    For any $n\geq 1$, we have $h_{n+2} = h_{n+1} +2h_{n}$.
\end{proposition}
\begin{proof}
    First, suppose that $n$ is of the form $2r$. We can compute:
    \begin{align*}
        h_{2r+2} &= \sum_{t=2}^{r+2}2^t\left(3\binom{2r+3-t}{t-2} + 5\binom{2r+2-t}{t-2}\right)\\
        &= \sum_{t=2}^{r+2}2^t\left(3\binom{2r+2-t}{t-2} + 3\binom{2r+2-t}{t-3} + 5\binom{2r+1-t}{t-2} + 5\binom{2r+1-t}{t-3}\right)\\
        h_{2r+2}-h_{2r+1} &= \sum_{t=2}^{r+2} 2^t\left(3\binom{2r+2-t}{t-3} + 5\binom{2r+1-t}{t-3}\right) + 5\cdot 2^{r+2}\binom{r-1}{r}\\
        &= 2\sum_{t=2}^{r+1}2^t\left(3\binom{2r+1-t}{t-2} + 5\binom{2r-t}{t-2}\right) = 2h_{2r}
    \end{align*}
    Now, suppose that $n$ is of the form $2r-1$
    \begin{align*}
        h_{2r+1}-h_{2r} &= 3\cdot 2^{r+2}\binom{r-1}{r} + \sum_{t=2}^{r+2}3\cdot 2^t\binom{2r+1-t}{t-3} + \sum_{t=2}^{r+1}5\cdot 2^t\binom{2r-t}{t-3}\\
        &= 2 \sum_{t=2}^{r+1}3\cdot 2^t\binom{2r-t}{t-2} + 2\sum_{t=2}^r 5\cdot 2^t\binom{2r-1-t}{t-2} = 2h_{2r-1}.
    \end{align*}
\end{proof}
\begin{corollary}\label{cor:hn-cn}
    For all $n\geq 0$, we have $h_n=c_n$.
\end{corollary}
\begin{proof}
    By \cref{def: bracket}, $\langle 0,0,\alpha\rangle=4$ for all $0\neq\alpha\in \mathscr{k}$. The case for $n = 1, 2$ can be verified numerically from \cref{eqn: sumbracket m = 1}. The result then follows from induction and \cref{prop: dn rr for m = 1}.
\end{proof}

\subsection{Calculation of \texorpdfstring{$\ehk(R_\alpha)$}{e\_HK(R\_α)} \& \texorpdfstring{$s(R_\alpha)$}{s(R\_α)}}\label{subsection:ehk and s}
We are now prepared to present the conjectural values of $\ehk(R_\alpha)$ and $s(R_\alpha)$.
 \begin{theorem}\label{thm-dn}
 In the region $|w|<\frac{1}{2}$ we have 
 \begin{align*}
    \sum\limits_{n=0}^\infty h_n w^n&= \frac{2(w+5)}{(w+1)(1-2w)},\\
     \sum\limits_{n=0}^\infty d_{n,\alpha} w^n&=\frac{2^{m+3}w^{m+1}+2^{m+1}w^m-4w-2}{(1-2w)(2^m w^m-1)}
 \end{align*}
 In particular when $w=\frac{1}{16}$,
 \begin{align*}
     \sum\limits_{n=0}^\infty c_n\left(\frac{1}{16}\right)^n&=\frac{582}{119},\\
     \sum\limits_{n=0}^\infty d_{n,\alpha}\left(\frac{1}{16}\right)^n&=\frac{17\cdot 2^{3m+1}-20}{7\cdot 2^{3m}-7}.
 \end{align*}
 \end{theorem}
\begin{proof}
We omit the verification of the first sum. In the region $|w|<\frac{1}{2}$ we have
{\small
\begin{align*}
    \sum\limits_{n=0}^\infty d_{n}w^n&=\sum\limits_{n=0}^\infty d_{mn} w^{mn}+\sum\limits_{l=1}^{m-1}\sum\limits_{n=0}^\infty d_{mn+l}w^{mn+l}\\
    &=4+8\cdot\sum\limits_{n=1}^\infty 2^{mn}w^{mn}+\sum\limits_{l=1}^{m-1}\sum\limits_{n=0}^\infty 3\cdot 2^{mn+l+1}w^{mn+l}\\
    &=4+8\cdot\frac{2^m w^m}{1-2^m w^m}+\frac{3\cdot 2^{m+1} w^m-12w}{(1-2w)(2^m w^m-1)}\\
    &=\frac{2^{m+3}w^{m+1}+2^{m+1}w^m-4w-4}{(1-2w)(2^m w^m-1)}.
\end{align*}
}
\end{proof}

 \begin{theorem}\label{thm:F-sig}
For $0\neq\alpha\in \mathscr{k}$ let $R_\alpha=\frac{\mathscr{k}\llbracket x,y,z,u,v\rrbracket}{(uv+g_\alpha)}$. If Conjecture \ref{conj:bracket} is true, then:
 \begin{enumerate}
    \item For $1\leq n<m_\alpha$, $e_n(G_\alpha)=\frac{45}{28}\cdot 2^{4n}-\frac{3}{7}\cdot 2^{n+1}$;\label{hk1}
     \item{\makebox[2cm][l]{$\ehk(R_\alpha)$}$=\begin{cases}\frac{767}{476}:&m_\alpha=1\\ \frac{45\cdot 2^{3m_\alpha}-38}{7\cdot 2^{3m_\alpha+2}-28}:&m_\alpha>1\end{cases}$}\label{hk2}
    \item{\makebox[2cm][l]{$s(R_\alpha)$}$=\begin{cases}\frac{185}{476}:&m_\alpha=1\\ \frac{11\cdot 2^{3m_\alpha}-18}{7\cdot 2^{3m_\alpha+2}-28}:&m_\alpha>1\end{cases}$}\label{hk3}
 \end{enumerate}
 \end{theorem}
 \begin{proof}
 We first remark that the relevant quantities are unchanged by completion, so the results from the previous two subsections apply. (\ref{hk1}) follows by induction combined with \cref{thm:difference-en} and \cref{lemma:recurrence} since 
\begin{align*}
    e_1(G_\alpha)=\dim_{\mathscr{k}}\frac{\mathscr{k}[x,y,z,u,v]}{(uv,x^2,y^2,z^2,u^2,v^2)}=24
\end{align*}
and $$3\cdot 2^{n+2}+16\left(\frac{45}{28}\cdot 2^{4n}-\frac{3}{7}\cdot 2^{n+1}\right)=\frac{45}{28}\cdot 2^{4n+4}-\frac{3}{7}2^{n+2}.$$
(\ref{hk2}): If $m_\alpha>1$, iterating (\ref{eqn:recurrence}) combined with \cref{thm:difference-en} yields
\begin{align}
    \frac{e_{n+1}}{2^{4(n+1)}}=e_0+\frac{1}{8}\sum\limits_{j=0}^n\frac{d_{j,\alpha}}{2^{4j}}.
\end{align}
Taking $n\rightarrow\infty$  and using \cref{thm-dn} we obtain 
\begin{align}
\ehk(R_\alpha)=1+\frac{1}{8}\sum\limits_{n=0}^\infty d_n\left(\frac{1}{16}\right)^n=\frac{45\cdot 2^{3m_\alpha}-38}{7\cdot 2^{3m_\alpha+2}-28}.\label{eqn:thm-F-sig-1}
\end{align}
The case of $m_\alpha=1$ follows identically using \cref{cor:hn-cn}. (\ref{hk3}): Let $\fm=\langle x,y,z,u,v\rangle\subseteq R_\alpha$ and consider the parameter ideal $I=\langle x,y,z,u+v\rangle\subseteq R_\alpha$. The image of $u$ generates the socle of $R_\alpha/I$ and $\langle I,u\rangle =\fm$. It follows that
\begin{align*}
    s(R_\alpha)=\ehk(I)-\ehk(\fm)=\ell_{R_\alpha}(R_\alpha/I)-\ehk(R_\alpha)=2-\ehk(R_\alpha)
\end{align*}
where the first equality follows from \cite[Thm. 6.2]{Hun12} (see also \cite[Thm. 11(2)]{HL02}) and the second equality holds because $R_\alpha$ is Cohen--Macaulay. The stated formula for $s(R_\alpha)$ then follows from part (\ref{hk2}).
\end{proof}

\begin{remark}
We emphasize a distinction between the formulas in \cref{thm:F-sig} and the main theorem of \cite{Mon98}. For any $0\neq \alpha\in \mathscr{k}$, writing $\alpha=\lambda^2+\lambda$, we have that
\begin{align*}
    \ehk\left(T_\alpha:=\frac{\mathscr{k}\llbracket x,y,z\rrbracket}{(g_\alpha)}\right)=3+4^{-m_\lambda}
\end{align*}
depends on $m_\lambda$ rather than on $m_\alpha$ as in \cref{thm:F-sig} (note that our notation $m_\alpha$ conflicts with that of \emph{op. cit.}). For example, let $\beta\in \mathscr{k}$ such that $\beta^6=\beta^4+\beta^3+\beta+1$, and let $\alpha=\beta^3+\beta$. The $\lambda\in \mathscr{k}$ such that $\alpha=\lambda^2+\lambda$ has $m_\lambda=12$, while $m_\alpha=6$. However, the $\lambda'$ for which $\beta=\lambda'^2+\lambda'$ is such that $m_{\lambda'}=m_\beta=6$. \cite{Mon98} then asserts that
\begin{align*}
\ehk(T_\alpha)=3+4^{-12}\neq 3+4^{-6}=\ehk(T_\beta).
\end{align*}
By contrast, \cref{thm:F-sig}  conjectures that
\begin{align*}
    \ehk(R_\alpha)=\frac{45\cdot 2^{18}-38}{7\cdot 2^{20}-28}=\ehk(R_\beta).
\end{align*}
\end{remark}

\section{Hilbert--Kunz Series and a multi--parameter family}\label{sec:two-parameter}
Throughout this section fix an algebraic closure $\mathscr{k}=\overline{\F_2}$ and let $f\in \mathscr{k}\llbracket x_1,\dots, x_r\rrbracket $. The \emph{Hilbert--Kunz series of $f$} is defined as
\begin{align*}
    \hks(f):=\sum\limits_{n=0}^\infty e_n(f)w^n\in\Z\llbracket w\rrbracket .
\end{align*}
Using this terminology, we may rephrase \cref{eqn:thm-F-sig-1} via the equation
\begin{equation}
    (1-16w)\hks(uv+g_\alpha)=1+2w\sum\limits_{n=0}^\infty d_{n,\alpha} w^n\label{hks-conjecture}
\end{equation}
whenever \Cref{conj:bracket} is true. One sees that the power series $(1-2^{r-1}w)\hks(f)$ converges in the region $|w|\leq\frac{1}{2^r}$ and when evaluated at $\frac{1}{2^{r+1}}$ is precisely the Hilbert--Kunz multiplicity of the hypersurface defined by $f$. Moreover, we can extract the Hilbert--Kunz multiplicity of complicated hypersurfaces from simpler ones using the following machine developed by Monsky in \cite{Mon09}, at least when $\Char \mathscr{k}=2$.

\begin{theorem}\cite[Theorem 1.6]{Mon09}\label{theorem-hadamard}
Let $f_1\in \mathscr{k}\llbracket x_{1,1},\dots,x_{1,r_1}\rrbracket ,\dots,f_t\in \mathscr{k}\llbracket x_{t,1},\dots,x_{t,r_t}\rrbracket $ be a finite collection of $t$ power series over $k$ in pairwise disjoint collections of variables. Then viewing $uv+\sum\limits_{i=1}^t f_i$ as a power series in $\mathscr{k}\llbracket u,v,x_{a,b}\mid 1\leq a\leq t,1\leq b\leq r_t\rrbracket $ we have
\begin{align*}
    \left(1-2^{1+\sum\limits_{i=1}^t r_i}w\right)\hks\left(uv+\sum\limits_{i=1}^t f_i\right)=\bigodot\limits_{i=1}^t \left(1-2^{1+r_i}w\right)\hks(uv+f_i)
\end{align*}
where $\bigodot$ denotes the Hadamard product.
\end{theorem}
\noindent Recall that the Hadamard product of two power series is given by
\begin{align*}
    \left(\sum\limits_{i=0}^\infty a_n w^n\right)\odot \left(\sum\limits_{i=0}^\infty b_n w^n\right):=\sum\limits_{i=0}^\infty a_n b_n w^n.
\end{align*}
From this, we are able to propose a conjectural formula for $\ehk(uv+\sum g_{\alpha_i})$ where the $g_{\alpha_i}$ are in pairwise disjoint sets of indeterminates.

\begin{theorem}\label{thm: multi-parameter}
Let $0\neq \alpha_1,\ldots,\alpha_t\in \mathscr{k}$ have degrees $m_1,\dots, m_t>1$ over $\F_2$, respectively. For each $i$, let $g_{\alpha_i}$ be the Brenner--Monsky quartic in the (pairwise disjoint) variables $x_i,y_i,z_i$. If \cref{conj:bracket} is true, then the $3t+1$-dimensional hypersurface $R_{\alpha_1,\ldots\alpha_n} := \frac{\mathscr{k}\llbracket x_i, y_i, z_i,u,v\mid 1\leq i\leq t\rrbracket}{(uv+g_{\alpha_1}+\cdots+g_{\alpha_t})}$ has Hilbert-Kunz multiplicity

\begin{align*}
\ehk(R_{\alpha_1,\ldots,\alpha_t})=\frac 32 + \frac{3^t}{2^{3t+2}-2^{t+1}}+\sum_{r=1}^t \sum_{i_1<\dots<i_r} \frac{3^{t-r}}{2^{(2t+1)\lcm(m_{i_1}, \dots, m_{i_r})+t+1}-2^{t+1}}
\end{align*}
and $F$-signature
\begin{align*}
    s(R_{\alpha_1,\ldots,\alpha_t})=\frac 12 - \frac{3^t}{2^{3t+2}-2^{t+1}}-\sum_{r=1}^t \sum_{i_1<\dots<i_r} \frac{3^{t-r}}{2^{(2t+1)\lcm(m_{i_1}, \dots, m_{i_r})+t+1}-2^{t+1}}
\end{align*}
where $i_j$ are integers between $1$ and $t$.
\end{theorem}
\begin{proof}
By \Cref{hks-conjecture} and \Cref{theorem-hadamard} we have that
\begin{align*}
    (1-2^{3t+1}w)\hks(uv+g)=1+8^t w+2^tw \sum\limits_{n=1}^\infty \pi_n w^{n}
\end{align*}
where $g = g_1+\dots+g_t$ and $\pi_n = d_{n,\alpha_1} d_{n,\alpha_2}\ldots d_{n,\alpha_t}$. To compute the sum on the right, note that we can express:
\[\pi_n = \begin{cases}
    4^t & n = 0\\
    3^{t-r}\cdot 2^{tn+t+2r} & n > 0 \text{ and exactly $r$ of the $m_i$ divide $n$.}
\end{cases}\]
Thus, we see that
\begin{align}
\sum_{n=1}^\infty\pi_n w^n &= \sum_{n=1}^\infty 3^t2^{tn+t}w^n + \sum_{r= 1}^t  \sum_{i_1<\ldots<i_r}\sum_{n=1}^\infty 3^{t-r}2^{tn \lcm(m_{i_1},\ldots,m_{i_r}) + t}w^{\lcm(m_{i_1},\ldots,m_{i_r})n}\label{eq: inclusion-exclusion}
\end{align}
where the sums converge as long as $|w|<2^{-t}$. To see this, fix an $n$ and suppose that $n$ divides exactly $r$ of the $m_i$. The $n$th coefficient of the right hand side of \cref{eq: inclusion-exclusion} is given by
\begin{align*}
    \sum_{j=0}^r \binom{r}{j}3^{t-j}2^{tn+t} &= 3^{t-r}2^{tn+t}\sum_{j=0}^r \binom{r}{j}3^{r-j} = 3^{t-r}2^{tn+t+2r}
\end{align*}
which is exactly $\pi_n$. We then see that
\begin{align*}
\sum_{n=1}^\infty\pi_n w^n &= \frac{12^tw}{1-2^tw}+\sum_{r=1}^t  \sum_{i_1<\ldots<i_r} \frac{3^{t-r}2^{t \lcm(m_{i_1},\ldots,m_{i_r}) + t} w^{\lcm(m_{i_1},\ldots,m_{i_r})}}{1-2^{t \lcm(m_{i_1},\ldots,m_{i_r})} w^{\lcm(m_{i_1},\ldots,m_{i_r})}}
\end{align*}
for all $w$ satisfying $|w|<2^{-t}$. Consequently, we find that
{\small
\begin{align*}
(1-2^{3t+1}w)\hks(uv+g) = 1 + 8^tw + \frac{24^tw^2}{1-2^tw}+ \sum_{r=1}^t  \sum_{i_1<\ldots<i_r} \frac{3^{t-r}2^{t \lcm(m_{i_1},\ldots,m_{i_r}) + 2t} w^{\lcm(m_{i_1},\ldots,m_{i_r}) + 1}}{1-2^{t \lcm(m_{i_1},\ldots,m_{i_r})} w^{\lcm(m_{i_1},\ldots,m_{i_r})}}.
\end{align*}
}
Evaluating at $w = 2^{-3t-1}$ gives the desired formula. The identity $s(R_{\alpha_1,\ldots,\alpha_t})=2-\ehk(R_{\alpha_1,\ldots,\alpha_t})$ follows \emph{mutatis mutandis} as \cref{thm:F-sig}: we note that $$I=\langle x_i,y_i,z_i,u+v\mid 1\leq i\leq t\rangle\subseteq R_{\alpha_1,\ldots,\alpha_t}$$ is generated by a system of parameters, $u$ generates the socle of $R_{\alpha_1,\ldots,\alpha_t}/I$, and $\langle I,u\rangle=\fm$ (the maximal ideal of $R_{\alpha_1,\ldots,\alpha_t}$). Since $R_{\alpha_1,\ldots,\alpha_t}$ is a Cohen--Macaulay local ring of multiplicity two, this concludes the proof.
\end{proof}

\section{\texorpdfstring{$F$}{F}-signature of Pairs}
We briefly recall (one instance of) the adaptation of the $F$-signature function to the pairs setting introduced in \cite{BST12,BST13}. Given a $d$-dimensional $F$-finite local ring $(R,\fm)$ of prime characteristic $p>0$ and $\fa\subseteq R$ a nonzero ideal, the \emph{$F$-signature} of the pair $(R,\fa^t)$ (where $t\in\R_{\geq 0}$) is given by
\begin{align*}
    s(R,\fa^t)=\lim\limits_{e\rightarrow\infty}\frac{\ell_R(R/(I_e:\fa^{\lceil tp^e\rceil}))}{p^{ed}}.
\end{align*}
Specializing to the case where $R$ is regular, $\fa=(f)$ is principal and $t=\frac{a}{p^c}$ is a rational number whose denominator is a power of $p$, this value may be conveniently expressed as a single length \cite[Proposition 4.1]{BST13}:
\begin{align*}
    s(R,f^{a/p^c})=\frac{\ell_R(R/(\fm^{[p^c]}:f^a))}{p^{cd}}.
\end{align*}
Even when $t\geq 0$ is not of this form but when $(R,\fm)$ is still regular and $\fa=(f)$ is principal, the function $t\mapsto s(R,f^t)$ enjoys the following interesting properties:
\begin{enumerate}
    \item $t\mapsto s(R,f^t)$ is convex on $[0,\infty)$ \cite[Theorem 3.5]{BST13};
    \item the left and right derivatives $\partial_- s(R,f^t)$ and $\partial_+ s(R,f^t)$ exist on $t\in (0,\infty)$, and the right derivative $\partial_+ s(R,f^t)$ exists at $t=0$ \cite[Corollary 3.6]{BST13};
    \item $\partial_- s(R,f^1)=-s(R/(f))$ and $\partial_+ s(R,f^0)=-\ehk(R/(f))$ \cite[Theorem 4.4]{BST13}.
\end{enumerate}
We include in this section some approximations of the plots of the function $$t\mapsto s(\mathscr{k}\llbracket x,y,z\rrbracket,g_\alpha^t)$$ and its derivative obtained via Macaulay2. The graph of the derivative $\frac{ds}{dt}$ is approximated via 
\begin{align*}
\frac{ds}{dt}\left(\frac{a}{p^c}\right)\approx\frac{s(R,f^{\frac{a+1}{p^c}})-s(R,f^{\frac{a}{p^c}})}{\frac{1}{p^c}}
\end{align*}
as in \cite[Section 4]{BST13}. The differences are quite subtle as $m_\alpha$ varies -- we must zoom in considerably to fully appreciate this phenomenon.

\begin{figure}[htp!]
\caption{}
\begin{tikzpicture}[scale=0.795,spy using outlines={circle, magnification=113, connect spies}]
\begin{axis}[
    axis lines = left,
    xlabel={$t$},
    ylabel={$s(\mathscr{k}\llbracket x,y,z\rrbracket,g_\alpha^t)$},
    xmin=0, xmax=0.5,
    ymin=0, ymax=1,
    legend pos=north east,
    ymajorgrids=true,
    grid style=dashed,
]

\addplot[
    color=RoyalBlue,
    mark=none,
    line width=.005pt
    ]
    coordinates {
(0, 1)
(.0078125, .97644)
(.015625, .952881)
(.0234375, .929321)
(.03125, .905762)
(.0390625, .882202)
(.046875, .858643)
(.0546875, .835083)
(.0625, .811523)
(.0703125, .788937)
(.078125, .766388)
(.0859375, .7439)
(.09375, .721436)
(.101562, .699215)
(.109375, .677032)
(.117188, .654911)
(.125, .632812)
(.132812, .612303)
(.140625, .591812)
(.148438, .571352)
(.15625, .550903)
(.164062, .530577)
(.171875, .510269)
(.179688, .489992)
(.1875, .469727)
(.195312, .449949)
(.203125, .430191)
(.210938, .410463)
(.21875, .390747)
(.226562, .371153)
(.234375, .351578)
(.242188, .332033)
(.25, .3125)
(.257812, .300659)
(.265625, .288818)
(.273438, .276978)
(.28125, .265137)
(.289062, .253296)
(.296875, .241455)
(.304688, .229614)
(.3125, .217773)
(.320312, .206905)
(.328125, .196075)
(.335938, .185307)
(.34375, .174561)
(.351562, .164059)
(.359375, .153595)
(.367188, .143192)
(.375, .132812)
(.382812, .124022)
(.390625, .11525)
(.398438, .106508)
(.40625, .0977783)
(.414062, .0891705)
(.421875, .0805817)
(.429688, .0720234)
(.4375, .0634766)
(.445312, .055418)
(.453125, .0473785)
(.460938, .0393696)
(.46875, .0313721)
(.476562, .0234966)
(.484375, .0156403)
(.492188, .00781441)
(.5, 0)
    };
  \addlegendentry{\(m_\alpha=3\)}
  
\addplot+[
color = BrickRed,
mark = none,
line width=.005pt]
 coordinates {
 (0, 1)
(.0078125, .976532)
(.015625, .953064)
(.0234375, .929596)
(.03125, .906128)
(.0390625, .882908)
(.046875, .859711)
(.0546875, .836582)
(.0625, .813477)
(.0703125, .790773)
(.078125, .768082)
(.0859375, .745424)
(.09375, .722778)
(.101562, .700258)
(.109375, .67775)
(.117188, .655275)
(.125, .632812)
(.132812, .612274)
(.140625, .591736)
(.148438, .571198)
(.15625, .550659)
(.164062, .530369)
(.171875, .510101)
(.179688, .489902)
(.1875, .469727)
(.195312, .449953)
(.203125, .430191)
(.210938, .410463)
(.21875, .390747)
(.226562, .371157)
(.234375, .351578)
(.242188, .332033)
(.25, .3125)
(.257812, .300751)
(.265625, .289001)
(.273438, .277252)
(.28125, .265503)
(.289062, .254002)
(.296875, .242523)
(.304688, .231113)
(.3125, .219727)
(.320312, .208742)
(.328125, .197769)
(.335938, .186831)
(.34375, .175903)
(.351562, .165102)
(.359375, .154312)
(.367188, .143557)
(.375, .132812)
(.382812, .123993)
(.390625, .115173)
(.398438, .106354)
(.40625, .0975342)
(.414062, .0889626)
(.421875, .0804138)
(.429688, .0719337)
(.4375, .0634766)
(.445312, .0554218)
(.453125, .0473785)
(.460938, .0393696)
(.46875, .0313721)
(.476562, .0235004)
(.484375, .0156403)
(.492188, .00781441)
(.5, 0)};
\addlegendentry{\(m_\alpha=4\)}

\addplot+[
color=green,
mark=none,
line width=.005pt]
 coordinates {(0, 1)
(.0078125, .976555)
(.015625, .95311)
(.0234375, .929729)
(.03125, .906372)
(.0390625, .88312)
(.046875, .859879)
(.0546875, .836672)
(.0625, .813477)
(.0703125, .790764)
(.078125, .768051)
(.0859375, .745403)
(.09375, .722778)
(.101562, .700258)
(.109375, .67775)
(.117188, .655275)
(.125, .632812)
(.132812, .612297)
(.140625, .591782)
(.148438, .571331)
(.15625, .550903)
(.164062, .530581)
(.171875, .510269)
(.179688, .489992)
(.1875, .469727)
(.195312, .449944)
(.203125, .430161)
(.210938, .410442)
(.21875, .390747)
(.226562, .371157)
(.234375, .351578)
(.242188, .332033)
(.25, .3125)
(.257812, .300774)
(.265625, .289047)
(.273438, .277386)
(.28125, .265747)
(.289062, .254213)
(.296875, .242691)
(.304688, .231203)
(.3125, .219727)
(.320312, .208733)
(.328125, .197739)
(.335938, .18681)
(.34375, .175903)
(.351562, .165102)
(.359375, .154312)
(.367188, .143557)
(.375, .132812)
(.382812, .124016)
(.390625, .115219)
(.398438, .106487)
(.40625, .0977783)
(.414062, .0891743)
(.421875, .0805817)
(.429688, .0720234)
(.4375, .0634766)
(.445312, .0554123)
(.453125, .047348)
(.460938, .0393486)
(.46875, .0313721)
(.476562, .0235004)
(.484375, .0156403)
(.492188, .00781441)
(.5, 0)
};
\addlegendentry{\(m_\alpha=5\)}

\coordinate (spypoint) at (axis cs:.1965, .17);
\coordinate (zoom) at (axis cs:.2993, .2388);
\coordinate (magnifyglass) at (axis cs:.1,.15);
\node [above=5pt of zoom] {$113\times$};

\end{axis}
\spy [Black, size=1.7cm] on (spypoint)
in node[fill=white] at (magnifyglass);
\end{tikzpicture}
\begin{tikzpicture}[scale=0.795,spy using outlines={circle, connect spies}]
\begin{axis}[
    axis x line=right,
    axis y line = left,
    xlabel={$t$},
    ylabel={$\frac{ds}{dt}$},
    xmin=0, xmax=0.5,
    ymin=-3.5, ymax=0,
    legend pos=south east,
    legend style={nodes={scale=0.8, transform shape}},
    ymajorgrids=true,
    grid style=dashed,
]
\addplot[
    color=RoyalBlue,
    mark=none,
    line width=.005pt
    ]
    coordinates {(0, -3.01562)
(.015625, -3.01562)
(.03125, -3.01562)
(.046875, -3.01562)
(.0625, -2.88867)
(.078125, -2.87695)
(.09375, -2.8418)
(.109375, -2.83008)
(.125, -2.62402)
(.140625, -2.61816)
(.15625, -2.60059)
(.171875, -2.59473)
(.1875, -2.53027)
(.203125, -2.52441)
(.21875, -2.50684)
(.234375, -2.50098)
(.25, -1.51562)
(.265625, -1.51562)
(.28125, -1.51562)
(.296875, -1.51562)
(.3125, -1.38867)
(.328125, -1.37695)
(.34375, -1.3418)
(.359375, -1.33008)
(.375, -1.12402)
(.390625, -1.11816)
(.40625, -1.10059)
(.421875, -1.09473)
(.4375, -1.03027)
(.453125, -1.02441)
(.46875, -1.00684)
(.484375, -1.00098)
(.5, 0)
};
\addlegendentry{\(m_\alpha=3\)}

\addplot+[
color=BrickRed,
    mark=none,
    line width=.005pt
]
coordinates {
(0, -3.00391)
(.015625, -3.00391)
(.03125, -2.9707)
(.046875, -2.95898)
(.0625, -2.90527)
(.078125, -2.89941)
(.09375, -2.88184)
(.109375, -2.87598)
(.125, -2.62891)
(.140625, -2.62891)
(.15625, -2.5957)
(.171875, -2.58398)
(.1875, -2.53027)
(.203125, -2.52441)
(.21875, -2.50684)
(.234375, -2.50098)
(.25, -1.50391)
(.265625, -1.50391)
(.28125, -1.4707)
(.296875, -1.45898)
(.3125, -1.40527)
(.328125, -1.39941)
(.34375, -1.38184)
(.359375, -1.37598)
(.375, -1.12891)
(.390625, -1.12891)
(.40625, -1.0957)
(.421875, -1.08398)
(.4375, -1.03027)
(.453125, -1.02441)
(.46875, -1.00684)
(.484375, -1.00098)
(.5, 0)};
\addlegendentry{\(m_\alpha=4\)}

\addplot+[color=green,
    mark=none,
    line width=.005pt
]
coordinates {
(0, -3.00098)
(.015625, -2.99121)
(.03125, -2.97559)
(.046875, -2.96973)
(.0625, -2.90723)
(.078125, -2.89746)
(.09375, -2.88184)
(.109375, -2.87598)
(.125, -2.62598)
(.140625, -2.61621)
(.15625, -2.60059)
(.171875, -2.59473)
(.1875, -2.53223)
(.203125, -2.52246)
(.21875, -2.50684)
(.234375, -2.50098)
(.25, -1.50098)
(.265625, -1.49121)
(.28125, -1.47559)
(.296875, -1.46973)
(.3125, -1.40723)
(.328125, -1.39746)
(.34375, -1.38184)
(.359375, -1.37598)
(.375, -1.12598)
(.390625, -1.11621)
(.40625, -1.10059)
(.421875, -1.09473)
(.4375, -1.03223)
(.453125, -1.02246)
(.46875, -1.00684)
(.484375, -1.00098)
(.5, 0)};
\addlegendentry{\(m_\alpha=5\)}

\coordinate (spypoint) at (axis cs:.215, -1.9);
\coordinate (magnifyglass) at (axis cs:.1,-1.3);

\coordinate (spypoint2) at (axis cs:.2727, -1.615);
\coordinate (magnifyglass2) at (axis cs:.3,-2.2);

\coordinate (spypoint3) at (axis cs:.026, -3.02);
\coordinate (magnifyglass3) at (axis cs:.05,-2.3);
\end{axis}
\spy [Black, size=1.4cm,magnification=5] on (spypoint)
in node[fill=white] at (magnifyglass);

\spy [Black, size=1.1cm,magnification=50] on (spypoint2)
in node[fill=white] at (magnifyglass2);

\spy [Black, size=1.4cm,magnification=20] on (spypoint3)
in node[fill=white] at (magnifyglass3);
\node [above=5pt of spypoint] {$5\times$};
\node [above=0pt of spypoint2] {$50\times$};
\node [below=5pt of spypoint3] {$20\times$};
\end{tikzpicture}
\end{figure}

\printbibliography

\end{document}